\documentclass[12pt,a4paper,reqno]{amsart}

\usepackage{anysize}
\usepackage[latin1]{inputenc}
\usepackage{amsmath}
\usepackage{amssymb}
\usepackage{array}

\marginsize{2.5cm}{2.5cm}{2.5cm}{2.5cm}

\begin{document}

\newtheorem{observacion}{Remark}
\newtheorem{theorem}{Theorem}[section]
\newtheorem{corollary}[theorem]{Corollary}
\newtheorem{lemma}[theorem]{Lemma}
\newtheorem{proposition}[theorem]{Proposition}
\newtheorem{definition}[theorem]{Definition}
\newtheorem{example}[theorem]{Example}
\newtheorem{remark}[theorem]{Remark}
\newtheorem{notation}[theorem]{Notation}
\newtheorem{question}[theorem]{Question}
\newtheorem{claim}[theorem]{Claim}
\newtheorem{conjecture}[theorem]{Conjecture}

\title{Ramanujan-Sato-like series}

\author{Gert Almkvist}
\address{Institute of Algeraic Meditation, Fogdar\"{o}d 208 S-24333 H\"{o}\"{o}r, SWEDEN}
\email{gert.almkvist@yahoo.se}

\author{Jesús Guillera}
\address{Av.\ Ces\'areo Alierta, 31 esc.~izda 4$^\circ$--A, Zaragoza, SPAIN}
\email{jguillera@gmail.com}

\date{}


\maketitle

\begin{abstract}
Using the theory of Calabi-Yau differential equations we obtain all the parameters of Ramanujan-Sato-like series for $1/\pi^2$ as $q$-functions valid in the complex plane. Then we use these q-functions together with a conjecture to find new examples of series of non-hypergeometric type. To motivate our theory we begin with the simpler case of Ramanujan-Sato series for $1/\pi$.
\end{abstract}

\section{Introduction}

In his famous paper of 1914 S. Ramanujan published $17$ formulas for $1/\pi$ \cite{Ra}, all of hypergeometric form
\[ \sum_{n=0}^{\infty }\frac{(1/2)_{n}(s)_{n}(1-s)_{n}}{n!^3}(a+b n) \, z^{n}=\frac{1}{\pi}. \]
Here
$(c)_{n}=c(c+1)(c+2) \cdots (c+n-1)$
is the Pochhammer symbol, $s=1/2$, $1/3$, $1/4$ or $1/6$ and $z$, $b$, $a$ are algebraic numbers. The most impressive is
\begin{equation}\label{ramanujan4}
\sum_{n=0}^{\infty} \frac{\left(\frac12 \right)_n \left(\frac14 \right)_n \left(\frac34 \right)_n}{(1)_n^3} \frac{1}{99^{4n}}(26390n+1103)=\frac{9801 \sqrt{2}}{4 \pi},
\end{equation}
which gives $8$ decimal digits of $\pi$ per term. All the $17$ series were rigourously proved in 1987 by the Borwein brothers \cite{Bo}. Independently, the Borwein \cite{Bo} and the Chudnovsky brothers \cite{ChCh} studied and proved Ramanujan series of the form
\begin{equation}\label{ChBo}
\sum_{n=0}^{\infty} \frac{(6n)!}{(3n)! n!^3} (a+bn) z^n=\frac{1}{\pi}.
\end{equation}

The value of $z$ can be found in the following way: Let us take the Chudnovsky brothers series (of Ramanujan type with $s=1/6$)
\[
\sum_{n=0}^{\infty} \frac{(6n)!}{(3n)! \, n!^3} (10177+261702n) \frac{1}{(-5280^{3})^n}=\frac{880^{2}\sqrt{330}}{\pi}.
\]
The series
\[
w_0=\sum_{n=0}^{\infty }\frac{(6n)!}{(3n)! \, n!^3} \, z^n = \sum_{n=0}^{\infty} 12^{3n} \cdot \frac{\left(\frac12 \right)_n \left(\frac16 \right)_n \left(\frac56 \right)_n}{(1)_n^3} z^n
\]
satisfies the differential equation
\[ \left( \frac{}{} \! \! \theta^3  - 24z(2 \theta + 1)(6 \theta + 1)(6 \theta + 5) \right) w_0 = 0, \]
where $\theta =z\, d/dz$. A second solution is
\[ w_1=w_0 \ln z+744z + 562932z^2+570443360z^3+\cdots. \]
Define
\[ q=\exp (\frac{w_1}{w_0})=z+744z^2+750420z^3+872769632z^4+\cdots. \]
Then
\[ J(q)=\frac{1}{z(q)}=\frac{1}{q}+744+196884q+21493760q^{2}+\cdots \]
is the famous modular invariant and
\[ J \left( \! -e^{-\pi \sqrt{67}} \, \right)=-5280^{3}. \]
A similar construction, getting a different $J:=1/z$, can be made starting with any third order differential equation which is the symmetric square of a second order differential equation. This kind of series are called Ramanujan-Sato series for $1/\pi$ because T. Sato discovered the first example of this type, one involving the Apéry numbers \cite{chan2}.
\par Similarly, the first formulas for $1/\pi^2$, found by the second author, were of hypergeometric type, using a function
\[ w_0=\sum_{n=0}^{\infty }\frac{(1/2)_{n}(s_1)_{n}(1-s_1)_{n}(s_2)_{n}(1-s_2)_n}{n!^5}\, z^n, \]
where the $14$ possible pairs $(s_1, s_2)$ are given in \cite{Gu5} or \cite{AlGu} and $w_{0}$ satisfies a fifth order differential equation
\[ \left(\theta^5-z(\theta +\frac12)(\theta +s_1)(\theta +1-s_1)(\theta +s_{2})(\theta +1-s_{2})\right) w_0=0, \]
This differential equation is of a very special type. It is a Calabi-Yau equation with a fourth order pullback with solutions $y_{0},y_{1},y_{2},y_{3}$, where
\[ w_0=y_0 \, (\theta y_1) - (\theta y_0) \, y_1. \]
This was used in \cite{Gu5} and \cite{AlGu}, where one new hypergeometric formula was found. Unfortunately, fifth order Calabi-Yau differential equations are quite rare. The simplest non-hypergeometric cases are Hadamard products of second and third order equations (labeled $A \ast \alpha$, etc in \cite{AlEnStZu}). Seven formulas of this kind have been found \cite{AlGu}, like for example
\begin{equation}\label{delta}
\sum_{n=0}^{\infty} \binom{2n}{n}^2 \sum_{i=0}^n \frac{(-1)^i 3^{n-3i} (3i)!}{i!^3} \binom{n}{3i} \binom{n+i}{i} \frac{(-1)^n}{3^{6n}} (803n^2+416n+68)=\frac{486}{\pi^2},
\end{equation}
which involves the Almkvist-Zudilin numbers. Two of the others were proved by Zudilin \cite{Zu1}. In this paper we explore more complicated fifth order equations, most of them found by the first author (\# 130 was found by H.Verrill).

\par To find $q_0$ in the $1/\pi$ case, we solve the equation $\alpha(q)=\alpha_0$, where $\alpha_0$ is a rational and
\[ \alpha(q)=\frac{\ln^2|q|}{\pi^2}. \]
The real solutions are $q_0=\pm e^{-\pi \sqrt{\alpha_0}}$. As there are many examples in the literature with $q_0$ real, in this paper we will show some series corresponding to $q_0=e^{i \pi r_0} e^{-\pi \sqrt{\alpha_0}}$, where $r_0$ is a rational such that $e^{i \pi r_0}$ is complex. If we calculate $J_0=J(q_0)$, then $z_0=1/J_0$.
In the $1/\pi^2$ case we have two functions
\begin{equation}\label{alpha-tau-intro}
 \alpha(q)=\frac{\frac{1}{6} \ln^3|q|-T(q)-h \zeta (3)}{\pi^2 \ln|q|}, \quad
\tau(q)=\frac{\frac{1}{2} \ln^2|q|-(\theta_q T)(q)}{\pi^2}-\alpha(q),
\end{equation}
where $h$ is an invariant and $T(q)$ essentially is the Gromov-Witten potential in String Theory. Solving the equation $\alpha(q)=\alpha_0$ numerically, where $\alpha_0$ is rational we get an approximation of $q_0$. Replacing $q_0$ in the second equation we get $\tau_0$. We conjecture that the corresponding series is of Ramanujan-type for $1/\pi^2$ if, and only if, $\tau_0^2$ is also rational. The success in finding the examples of this paper depends heavily on our experimental method to get the invariant $h$. It uses the critical value $z=z_c$, the radius of convergence for the power series $w_0$. From the conjecture $(dz/dq)(q_c)=0$ we get an approximation of $q_c$ and using the PSLQ algorithm to find an integer relation among the numbers
\[ \frac{\ln^3 |q_c|}{6}-T(q_c), \qquad \pi^2 \ln |q_c|, \qquad \zeta(3), \]
we obtain simultaneously $\alpha_c$ and the invariant $h$. Replacing $\alpha_c$ and $q_c$ in the second equation we get $\tau_c$.
\par In the $1/\pi$ case, the algebraic but non-rational $z_0$ dominate the rational solutions (see the tables in \cite{Al6}). In the case $1/\pi^2$ the only known series with a non-rational $z_0$ is
\begin{multline}\label{f9}
\sum_{n=0}^{\infty}  \frac{\left(\frac12 \right)_n^3 \left( \frac13 \right)_n \left( \frac23 \right)_n}{(1)_n^5} \left( \frac{15 \sqrt{5}-33}{2} \right)^{3n} \times \\
\left[ \frac{}{}(1220/3-180\sqrt{5})n^2+(303-135\sqrt{5})n+(56-25\sqrt{5})\right]=\frac{1}{\pi^2},
\end{multline}
which was discovered by the second author \cite{Gu7}. See also the corresponding mosaic supercongruences in \cite{Gu6}.
\par We obtain the $q$-functions for all the parameters of general Ramanujan-Sato-like series for $1/\pi$ and $1/\pi^2$. Contrary to the series for $1/\pi$ in which everything can be proved rigourously by means of modular equations, in the case $1/\pi^2$ we can only evaluate the functions numerically and then guess the algebraic values when they exist. A modular-like theory which explains the rational and algebraic quantities observed is still not available \cite{yangzudilin}. For an excellent account of these topics, see \cite{Zu4}.

\section{Ramanujan-Sato-type series for $1/\pi$}

Certain differential equations of order $3$ are the symmetric square of a differential equation of second order. Suppose
\begin{equation}\label{ecudif-order3}
\theta^3 w = e_2(z) \theta^2 w + e_1(z) \theta w +e_0(z)w, \qquad \theta=z \frac{d}{dz},
\end{equation}
is the symmetric square of the second order equation
\begin{equation}\label{ecudif-order2}
\theta^2 y = c_1(z) \theta y +c_0(z)y, \qquad 3c_1(z)=e_2(z).
\end{equation}
We define the following function:
\[ P(z)=\exp \int \frac{-2 c_1(z)}{z} dz, \]
with $P(0)=1$, which plays an important role in the theory. In the examples of this paper $P(z)$ is a polynomial but we have also found examples for which $P(z)$ is a rational function.
\par The fundamental solutions $w_0$, $w_1$, $w_2$ of the third order differential equation are connected to the fundamental solutions $y_0$, $y_1$ of the second order equation by
\begin{equation}\label{prop-wy-simple}
w_0=y_0^2, \quad w_1=y_0 y_1, \quad w_2=\frac12 \, y_1^2,
\end{equation}
\cite[Prop. 9]{AlZu}. We define the wronskians
\[
W(w_i,w_j)=\left|
  \begin{array}{cc}
    w_i & \theta w_i \\
    w_j & \theta w_j \\
  \end{array}
\right|, \qquad
W(y_0,y_1)=\left|
  \begin{array}{cc}
    y_0 & \theta y_0 \\
    y_1 & \theta y_1 \\
  \end{array}
\right|.
\]
Observe that this notation is not the same as in \cite{Al2}, where in the definition of $W(y_0,y_1)$ we have $y_0'$ and $y_1'$ instead of $\theta y_0$ and $\theta y_1$.
\begin{theorem}
We have
\begin{equation}\label{wronsk-simple}
W(w_0,w_1)=\frac{y_0^2}{\sqrt{P}}, \quad
W(w_0,w_2)=\frac{y_0y_1}{\sqrt{P}}, \quad
W(w_1,w_2)=\frac{y_1^2}{2\sqrt{P}}.
\end{equation}
\end{theorem}
\begin{proof}
Using (\ref{prop-wy-simple}), we get
\[ W(w_0,w_1)=y_0^2 \, W(y_0,y_1), \quad W(w_0,w_2)=y_0y_1 \, W(y_0,y_1) \]
and
\[ W(w_1,w_2)=\frac12 y_1^2 \, W(y_0,y_1). \]
If we denote with $f$ the wronskian $W(y_0,y_1)$, then from (\ref{ecudif-order2}) we see that $\theta f = c_1(z) f$. This implies $f=1/\sqrt{P}$.
\end{proof}

\subsection{Series for $1/\pi$}

Let $q=e^{i \pi r} e^{-\pi \tau}$. If the function
\[ w_0(z)=\sum_{n=0}^{\infty} A_n z^n, \]
satisfies a differential equation of order $3$ as above, then we will find two functions $b(q)$ and $a(q)$ with good arithmetical properties, such that
\[ \sum_{n=0}^{\infty} A_n \left(\frac{}{} \! \! a(q)+b(q)n\right)z^n(q)=\frac{1}{\pi}. \]
The interesting cases are those with $z$, $b$, $a$ algebraic. They are called Ramanujan-Sato-type series for $1/\pi$.
\par The usual $q$-parametrization is
\[ q=\exp(\frac{y_1}{y_0})=\exp(\frac{w_1}{w_0}), \]
and we can invert it to get $z$ as a series of powers of $q$. The function $z(q)$ is the mirror map and for this kind of differential equations it has been proved that it is a modular function. We also define $J(q):=1/z(q)$.
\begin{theorem}
The functions $\alpha(q)$, $b(q)$, $a(q)$ such that
\begin{equation}\label{expansion-simple}
\sum_{j=0}^{2} \left[ \frac{}{} \! \! (w_j) a + (\theta w_j) b \right] x^j=e^{i \pi r x} \left( \frac{1}{\pi} - \frac{\pi}{2} \, \alpha x^2 \right) \, \, \text{\rm truncated at} \, \, x^3
\end{equation}
are given by
\begin{equation}\label{alpha-b-a-simple}
\alpha(q)=\tau^2(q), \quad b(q)=\tau(q) \sqrt{P(z)}, \quad a(q)=\frac{1}{\pi w_0}\left( 1+\frac{\ln|q|}{w_0} \, q \frac{dw_0}{dq} \right).
\end{equation}
In addition, if $r$ and $\tau_0^2$ are rational then $z(q_0)$, $b(q_0)$, $a(q_0)$ are algebraic.
\end{theorem}
\begin{proof}
First, we see that $q=e^{i \pi r} e^{-\pi \tau}$ implies that
\[ \tau(q)=-\frac{\ln|q|}{\pi}. \]
We can write (\ref{expansion-simple}) in the following equivalent form:
\begin{align}
(w_0) \, a + (\theta w_0) \, b &=\frac{1}{\pi}, \nonumber \\
(w_1) \, a + (\theta w_1) \, b &=i r, \label{sistema-simple} \\
(w_2) \, a + (\theta w_2) \, b &=-\frac{\pi}{2} ( \alpha + r^2).  \nonumber
\end{align}
In what follows we will use the wronskians (\ref{wronsk-simple}). As we want this system to be compatible, we have
\begin{equation}\label{deter-eq0}
\left|
  \begin{array}{cccc}
    w_0 & \theta w_0 & \frac{1}{\pi}  \\
    w_1 & \theta w_1 & i r  \\
    w_2 & \theta w_2 & -\frac{\pi}{2} (\alpha + r^2)  \\
  \end{array}
\right|=0.
\end{equation}
Expanding along the last column, we get
\[ \frac{1}{2\pi} \left( \frac{y_1}{y_0} \right)^2 -i r \left( \frac{y_1}{y_0} \right) - \frac{\pi}{2}(\alpha+r^2)=0. \]
Hence
\[ \frac{1}{\pi} \frac{\ln^2 q}{2}-i r \ln q -\frac{\pi}{2}(\alpha+r^2)=0. \]
As $\ln q =\ln|q|+i \pi r$ we obtain the function $\alpha(q)$. To obtain $b$ we apply Cramer's method to the system formed by the two first equations of (\ref{sistema-simple}). We get
\[
b=\left( i r - \frac{1}{\pi} \frac{w_1}{w_0} \right)\sqrt{P(z)}=
\left( \frac{i \pi r-\ln q}{\pi}\right)\sqrt{P(z)}=-\frac{\ln|q|}{\pi}\sqrt{P(z)}.
\]
Then, replacing $w_1$ with $w_0 \ln q$ in the second equation of (\ref{sistema-simple}) and solving the system formed by the two first equations, we obtain the identity
\[ w_0=\frac{q}{z\sqrt{P(z)}}\frac{dz}{dq}. \]
Finally, using the two last formulas and the first equation of (\ref{sistema-simple}), we derive the formula for $a(q)$ in (\ref{alpha-b-a-simple}). From $b=\tau \sqrt{1-z}$ we see that $b$ takes algebraic values when $r$ and $\tau^2$ are rational. By an analogue to the argument given in \cite[Sect. 2.4]{Gu5} we see that the same happens to $a(q)$.
\end{proof}

\subsection{Examples of series for $1/\pi$}

There are many examples in the literature (see \cite{BaBeCh} and the references in it) but until very recently all of them were with $r=0$ (series of positive terms) or with $r=1$ (alternating series). The first example of a complex series was found and proved, with a hypergeometric transformation, by the second author and Wadim Zudilin in \cite[eq. 44]{GuiZu}. Other complex series, proved by modular equations or hypergeometric transformations, are in  \cite{ChWaZu}, like for example
\begin{equation}
\sum_{n=0}^{\infty} \frac{(4n)!}{n!^4} \left( \frac{10+2\sqrt{-3}}{28 \sqrt{3}} \right)^{4n} \left( \frac{}{} \! \! (320-55\sqrt{-3})n+(52-12\sqrt{-3}) \right) = \frac{98\sqrt{3}}{\pi}.
\end{equation}
Tito Piezas \cite{tito} found numerically and then guessed the series
\begin{equation}
\sum_{n=0}^{\infty} \frac{(2n)!(3n)!}{n!^5}\frac{3+(17-i)n}{(2(7+i)(2+i)^4)^n}=\frac{33-6i}{4}\, \frac{1}{\pi},
\end{equation}
which involves only gaussian rational numbers. It leads to taking $q=e^{\frac{2 \pi i}{3}} e^{-\frac{4\sqrt{2}}{3}\pi}$ and of course it is possible to prove it rigourously using modular equations. For our following examples we have chosen the sequence of numbers
\[ A_n=\binom{2n}{n} \sum_{k=0}^n \binom{n}{k} \binom{2k}{k} \binom{2n-2k}{n-k}, \]
which is the Hadamard product $\binom{2n}{n} * (d)$, (see \cite{AlZu}). The differential equation is
\[ \left( \frac{}{} \! \! \theta^3-8z(2\theta+1)(3\theta^2+3\theta+1)+128z^2(\theta+1)(2\theta+1)(2\theta+3) \right) w = 0. \]
The polynomial $P(z)$ is $P(z)=(1-16z)(1-32z)$ and
\[ J(q)=\frac{1}{z(q)}=\frac{1}{q}+16+52q+834q^3+4760q^5+24703q^7+\cdots. \]
For $q=i e^{-\pi\frac{\sqrt{13}}{2}}$, we find
\begin{equation}
\sum_{n=0}^{\infty} A_n \frac{(-1+6i)+(-9+33i)n}{(16+288i)^n}=\frac{52+91i}{\sqrt{13 (1+18i)^3}} \, \cdot \frac{50}{\pi}.
\end{equation}
For $q=i e^{-\pi \frac{\sqrt{37}}{2}}$, we get
\begin{equation}
\sum_{n=0}^{\infty} A_n \frac{(11842+11741i)+112665(1+i)n}{(16-14112i)^n}= \left( \frac{37}{1+882i} \right)^{\frac32} \cdot \frac{2 \cdot 5^3 \cdot 29^3}{\pi}.
\end{equation}
Taking $q=e^{i\frac{\pi}{4}} e^{-\pi\frac{\sqrt{15}}{4}}$, we find
\begin{equation}
\sum_{n=0}^{\infty} A_n \left[ \frac{}{} \! \! 4(-3\sqrt{3}+13i)+3(29\sqrt{3}-44i)n \right] \left( \frac{2}{65-15\sqrt{3} i} \right)^n = \frac{14^2 \sqrt{5}}{\pi}.
\end{equation}
We give a final example with the sequence of Domb's numbers \cite{chan2}, called $(\alpha)$ in \cite{AlZu}
\[ A_n=\sum_{k=0}^n \binom{n}{k}^2 \binom{2k}{k} \binom{2n-2k}{n-k}. \]
The differential equation is
\[ \left( \frac{}{} \! \! \theta^3-2z(2\theta+1)(5\theta^2+5\theta+2)+64z^2(\theta+1)^3 \right) w = 0. \]
We have $P(z)\!=\!(1-4z)(1-16z)$ and
\[ J(q) \!=\! q^{-1}+6+15q+32q^2+87q^3+192q^4+\cdots. \]
For $q=e^{\frac{i\pi}{3}}e^{-\frac{2 \sqrt{2}}{3}\pi}$, we find
\begin{equation}
\sum_{n=0}^{\infty} A_n \frac{(1+i)+(4+2i)n}{(16+16i)^n}=\frac{6}{\pi}.
\end{equation}
Taking the real and imaginary parts, we get
\[
\sum_{n=0}^{\infty} A_n \frac{1+4n}{32^n} (\sqrt{2})^n \cos \frac{n\pi}{4}
+ \sum_{n=0}^{\infty} A_n \frac{1+2n}{32^n} (\sqrt{2})^n \sin \frac{n\pi}{4} = \frac{6}{\pi},
\]
and
\[
\sum_{n=0}^{\infty} A_n \frac{1+2n}{32^n} (\sqrt{2})^n \cos \frac{n\pi}{4}
= \sum_{n=0}^{\infty} A_n \frac{1+4n}{32^n} (\sqrt{2})^n \sin \frac{n\pi}{4}.
\]
The first author is preparing a collection of series for $1/\pi$ in \cite{Al6}. Although we have guessed our examples from numerical approximations, the exact evaluations can be proved rigourously by using modular equations \cite{ChWaZu}.

\section{Ramanujan-Sato-like series for $1/\pi^2$}

A Calabi-Yau differential equation is a $4^{th}$ order differential equation
\begin{equation}\label{calabi}
\theta^4 y = c_3(z) \theta^3 y + c_2(z) \theta^2 y + c_1(z) \theta y +c _0(z) y, \qquad \theta=z \frac{d}{dz},
\end{equation}
where $c_i(z)$ are quotients of polynomials of $z$ with rational coefficients, which satisfies several conditions \cite{AlGu}. There are two functions associated to these equations which play a very important role, namely: The mirror map and the Yukawa coupling. The mirror map $z(q)$ is defined as the functional inverse of
\[ q=\exp(\frac{y_1}{y_0}) \]
and the Yukawa coupling as
\[ K(q)=\theta_q^2 (\frac{y_2}{y_0}), \qquad \theta_q=q \frac{d}{dq}. \]
We define $T(q)$ as the unique power series of $q$ such that $T(0)=0$ and
\[ \theta^3_q \, T(q)=1-K(q). \]
The function
\begin{equation}\label{GW}
\Phi =\frac{1}{2}(\frac{y_{1}}{y_{0}}\frac{y_{2}}{y_{0}}-\frac{y_{3}}{y_{0}})=\frac{1}{6}\ln^3 q - T(q),
\end{equation}
is well-known in String Theory and is called the Gromov-Witten potential (see \cite[p.28]{CoKa}).

\subsection{Pullback}

The solution of some differential equations of $5^{th}$ order can be recovered from the solutions of a $4^{th}$ order Calabi-Yau differential equation. We say that they admit a pullback. If (\ref{calabi}) is the ordinary pullback of the differential equation
\begin{equation}\label{order5calabi}
\theta^5 w = e_4(z) \theta^4 w + e_3(z) \theta^3 w + e_2(z) \theta^2 w + e_1(z) \theta w + e_0(z) w,
\end{equation}
then we know that $w_0$, $w_1$, $w_2$, $w_3$, $w_4$ can be recovered from the $4$ fundamental solutions $y_0$, $y_1$, $y_2$, $y_3$ of (\ref{calabi}) in the following way
\begin{align}
w_0=\left|
  \begin{array}{cc}
    y_{0}  & y_{1}  \\
    \theta y_{0} & \theta y_{1} \\
  \end{array}
\right|, \, \,
w_1 &= \left|
  \begin{array}{cc}
    y_{0}  & y_{2}  \\
    \theta y_{0} & \theta y_{2} \\
  \end{array}
\right|, \, \,
w_3 =\frac{1}{2} \left|
  \begin{array}{cc}
    y_{1}  & y_{3}  \\
    \theta y_{1} & \theta y_{3} \\
  \end{array}
\right|, \, \,
w_4= \frac{1}{2} \left|
  \begin{array}{cc}
    y_{2}  & y_{3}  \\
    \theta y_{2} & \theta y_{3} \\
  \end{array}
\right|,  \label{single} \\
w_2 &=\left|
  \begin{array}{cc}
    y_{0}  & y_{3}  \\
    \theta y_{0} & \theta y_{3} \\
  \end{array}
\right|=
\left|
  \begin{array}{cc}
    y_{1}  & y_{2}  \\
    \theta y_{1} & \theta y_{2} \\
  \end{array}
\right|. \label{double}
\end{align}
We define the following function:
\[ P(z)=\exp \int \frac{-2 c_3(z)}{z} dz, \]
with $P(0)=1$, which plays an important role in the theory. In the Yifan Yang's pullback the corresponding coefficient is $4c_3(z)$ instead of $c_3(z)$. In all the examples of this paper $P(z)$ is a polynomial.
\par We denote as $W(w_i,w_j,w_j)$ and $W(w_i,w_j)$ the following wronskians \cite{Al2}:
\begin{equation}\label{wronskians}
W(w_i,w_j,w_k)=\left|
  \begin{array}{cccc}
    w_i & \theta w_i & \theta^2 w_i \\
    w_j & \theta w_j & \theta^2 w_j \\
    w_k & \theta w_k & \theta^2 w_k
  \end{array}
\right|, \qquad
W(w_i,w_j)=\left|
  \begin{array}{cccc}
    w_i & \theta w_i \\
    w_j & \theta w_j
  \end{array}
\right|.
\end{equation}
Due to different definition and notation $f$ in \cite{Al2} is $1/\sqrt[4]{P(z)}$ here and the powers of $x$ ($z$ here) do not appear now. We will need the following wronskians of order $3$, (see \cite{Al2}):
\[
W(w_1,w_2,w_3)=\frac12 \frac{y_1 y_2-y_0y_3}{\sqrt{P}}, \quad  W(w_0,w_2,w_3)=\frac{y_1^2}{\sqrt{P}}, \quad W(w_0,w_1,w_3)=\frac{y_0y_1}{\sqrt{P}},
\]
\[
W(w_0,w_1,w_2)=\frac{y_0^2}{\sqrt{P}}, \quad W(w_0,w_1,w_4)=\frac{y_0y_2}{\sqrt{P}}, \quad W(w_1,w_2,w_4)=\frac{y_2^2}{2\sqrt{P}},
\]
and
\[ W(w_0,w_2,w_4)=\frac{y_0y_3+y_1y_2}{2\sqrt{P}}. \]
We will also need the following wronskians of order $2$, (see \cite{Al2}):
\[
W(w_0,w_1)=\frac{y_0^2}{\sqrt[4]{P}}, \quad W(w_0,w_2)=\frac{y_0y_1}{\sqrt[4]{P}}, \quad W(w_1,w_2)=\frac{y_0y_2}{\sqrt[4]{P}}.
\]

\subsection{Series for $1/\pi^2$}
Suppose that the function
\begin{equation}\label{mainseries}
w_0(z)=\sum_{n=0}^{\infty} A_n z^n,
\end{equation}
is a solution of a $5^{th}$ order differential equation which has a pullback to a Calabi-Yau differential equation. We will determine functions $a(q)$, $b(q)$, $c(q)$ in terms of $\ln|q|$, $z(q)$ and $T(q)$, such that
\[ \sum_{n=0}^{\infty} A_n z(q)^n (a(q)+b(q)n+c(q)n^2)=\frac{1}{\pi^2}. \]
The interesting cases are those for which $z$, $c$, $b$, $a$ are algebraic numbers. We will call them Ramanujan-Sato-like series for $1/\pi^2$.
In this paper we improve and generalize to the complex plane the theory developed in \cite{AlGu} and \cite{Gu7}. We let $q=|q| \, e^{i\pi r}$, and consider an expansion of the form:
\begin{multline}\label{expansion}
\sum_{j=0}^{4} \left[ \frac{}{} \! \! (w_j) a + (\theta w_j) b + (\theta^2 w_j) c \right] x^j
\\= e^{i \pi r x} \left( \frac{1}{\pi^2}-\alpha x^2 + h \frac{\zeta(3)}{\pi^2}x^3+\frac{\pi^2}{2}(\tau^2-\alpha^2)x^4 \right) \, \, \text{truncated at} \, \, x^5.
\end{multline}
The number $h$ is a rational constant associated to the differential operator $D$ such that $Dw_0=0$. The motivation of this expansion is due to the fact that in the case of Ramanujan-Sato-like series for $1/\pi^2$ ($z$, $c$, $b$, $a$ algebraic), we have experimentally observed that $r$, $\alpha$ and $\tau^2$ are rational while $h$ is a rational constant (see the remark at the end of this section). We have the equivalent system
\begin{align}
(w_0) \, a + (\theta w_0) \, b + (\theta^2 w_0) \, c &= \frac{1}{\pi^2}, \nonumber \\
(w_1) \, a + (\theta w_1) \, b + (\theta^2 w_1) \, c &= \frac{i}{\pi} r, \nonumber \\
(w_2) \, a + (\theta w_2) \, b + (\theta^2 w_2) \, c &= -\frac{r^2}{2}-\alpha, \label{sistema} \\
(w_3) \, a + (\theta w_3) \, b + (\theta^2 w_3) \, c &= i\pi r \left( -\frac{r^2}{6}-\alpha \right) + h \frac{\zeta(3)}{\pi^2}, \nonumber \\
(w_4) \, a + (\theta w_4) \, b + (\theta^2 w_4) \, c &= \pi^2 \left(\frac{r^4}{24}+\frac{\tau^2-\alpha^2}{2}+\frac{r^2}{2} \alpha \right)+\frac{i}{\pi} h \zeta(3) r. \nonumber
\end{align}
This system allows us to develop the theory. In the next theorem we obtain $\alpha$ and $\tau$ as non-holomorphic functions of $q$.
\begin{theorem}\label{th-alfa-tau}
We have
\begin{equation}
\label{for-alpha} \alpha(q)=\frac{\frac{1}{6} \ln^3|q|-T(q)-h \zeta (3)}{\pi^2 \ln|q|},
\end{equation}
and
\begin{equation}
\label{for-tau} \tau(q)=\frac{\frac{1}{2} \ln^2|q|-(\theta_q T)(q)}{\pi^2}-\alpha(q).
\end{equation}
\end{theorem}
\begin{proof}
In the proof we use the wronskians above. As we want the system (\ref{sistema}) to be compatible, we have
\begin{equation}\label{deter-eq0}
\left|
  \begin{array}{cccc}
    w_0 & \theta w_0 & \theta^2 w_0 & p_0  \\
    w_1 & \theta w_1 & \theta^2 w_1 & p_1  \\
    w_2 & \theta w_2 & \theta^2 w_2 & p_2  \\
    w_3 & \theta w_3 & \theta^2 w_3 & p_3
  \end{array}
\right|=0,
\end{equation}
where $p_0$, $p_1$, etc. stand for the independent terms. Expanding the determinant along the last column, we obtain
\[ -p_0 \left(\frac{y_1 y_2-y_0y_3}{2}\right)+p_1 (y_1^2)-p_2 (y_0 y_1) + p_3 (y_0^2)=0. \]
Then, dividing by $y_0^2$, we get
\[
-p_0 \frac{1}{2} \left(\frac{y_1}{y_0} \frac{y_2}{y_0}-\frac{y_3}{y_0}\right)+p_1 \left(\frac{y_1}{y_0}\right)^2-p_2 \left(\frac{y_1}{y_0}\right) + p_3=0.
\]
Hence
\[
-p_0 \left( \frac16 \ln^3 q - T(q)  \right)+p_1 \ln^2 q - p_2 \ln q + p_3 = 0.
\]
Using $\ln q = \ln|q|+i \pi r$ and replacing $p_0$, $p_1$, $p_2$ and $p_3$ with there values in the system, we arrive at (\ref{for-alpha}). Then, from the first, second, third and fifth equations and using the the function $\alpha(q)$ obtained already, we derive (\ref{for-tau}).
\end{proof}
In the next theorem we obtain $c$, $b$, $a$ as non-holomorphic functions of $q$.
\begin{theorem}\label{teor-abc}
\begin{align}
\label{for-c}  c(q)   &=\tau(q) \sqrt[4]{P(z)}, \\
\label{for-b}  b(q)   &= \frac{z(q)}{\theta_q z(q)} \left( \frac{1}{\pi^2} \left(\frac{}{} \! \! \theta_q^2 T(q)-\ln|q|\right) - \tau(q) \, \frac{\theta_q L(q)}{L(q)} \right) \sqrt[4]{P(z)}, \\
\label{for-a}  a(q)   &=\frac{1}{w_0(q)}\left( \frac{1}{\pi^2} - (\theta w_0) b(q)- (\theta^2 w_0) c(q) \right),
\end{align}
with
\begin{equation}\label{L}
L(q)=\frac{y_0^2}{\sqrt[4]{P(z)}}=\frac{w_0(q)}{\sqrt[4]{P(z)}} \, \frac{\theta_q z(q)}{z(q)}= \frac{1}{\sqrt{P(q )}K(q)} \, \left(\frac{\theta_q z(q)}{z(q)}\right)^3,
\end{equation}
where $y_0$ is the ordinary pullback.
\end{theorem}
\begin{proof}
Solving for $c$ by Cramer's rule from the three first equations of (\ref{sistema}), we get
\[ \frac{c}{\sqrt[4]{P(z)}} = \frac{1}{\pi^2} \left( \frac{y_2}{y_0} \right) -\frac{i}{\pi}r \left( \frac{y_1}{y_0} \right) - \frac{r^2}{2} - \alpha. \]
Hence
\[ \frac{c}{\sqrt[4]{P(z)}} = \frac{1}{\pi^2} \left( \frac12 \ln^2 q - \theta_q T \right)-\frac{i}{\pi}r \ln q-\frac{r^2}{2}-\alpha. \]
Replacing $\ln q$ with $\ln|q|+i \pi r$, we obtain (\ref{for-c}).
Then, solving for $b$  from the two first equations of (\ref{sistema}), we obtain
\begin{equation}\label{b-inter}
b=\frac{1}{\pi^2} \, \frac{w_0}{L} \left( i \pi r - \frac{w_1}{w_0} \right) - c(z) \, \frac{\theta_z L}{L},
\end{equation}
where $L=w_0 (\theta w_1)-w_1 (\theta w_0)$. But, as $q=\exp(y_1/y_0)$, we obtain
\[
\theta_q (\frac{y_2}{y_0})=\frac{q}{\theta_z q} \theta_z (\frac{y_2}{y_0})=\frac{y_0 \, \theta y_2 - y_2 \, \theta y_0}{y_0 \,
\theta y_1 - y_1 \, \theta y_0}=\frac{w_1}{w_0}.
\]
Applying $\theta_q$ to the two extremes of it, we get
\[
\theta_q \left( \frac{w_1}{w_0} \right)=\frac{w_0 (\theta w_1)-w_1 (\theta w_0)}{w_0^2} \, \frac{\theta_q z}{z}=K(q)=1-\theta_q^3 T(q),
\]
which implies
\begin{equation}\label{a1-over-a0-T}
\frac{w_1}{w_0} = \ln q - \theta_q^2 T(q),
\end{equation}
and
\begin{equation}\label{a0b1-a1b0-L}
w_0 \theta w_1 - w_1 \theta w_0=\frac{w_0^2 K(q)z(q)}{\theta_q z(q)}=L(q).
\end{equation}
But
\begin{equation}\label{L-y0}
L=w_0 (\theta w_1) - w_1 (\theta w_0) = \frac{y_0^2}{\sqrt[4]{P(z)}},
\end{equation}
In \cite{Al2} we have the formula
\begin{equation}\label{y0-cua}
 y_0^2 = \left( \frac{\theta_q z(q)}{z(q)} \right)^3 \frac{1}{\sqrt[4]{P(q)} \, K(q)}.
\end{equation}
From (\ref{a0b1-a1b0-L}), (\ref{L-y0}) and (\ref{y0-cua}), we obtain
\begin{equation}\label{w0-cua}
 w_0  =  \left( \frac{\theta_q z(q)}{z(q)} \right)^2 \frac{1}{\sqrt[4]{P(q)} \, K(q)}.
\end{equation}
From the three last identities we arrive at (\ref{L}). From (\ref{b-inter}), (\ref{a1-over-a0-T}) and (\ref{w0-cua}) we deduce (\ref{for-b}). The proof of (\ref{for-a}) is trivial from the first equation of (\ref{sistema}).
\end{proof}

The relevant fact is that the functions $\alpha(q)$, $\tau(q)$, $c(q)$, $b(q)$, $a(q)$, have good arithmetical properties. This is stated in the following conjecture which is crucial to discover Ramanujan-Sato-like series for $1/\pi^2$:

\subsection*{Conjecture}

Let $\alpha_0=\alpha(q_0)$, $\tau_0=\tau(q_0)$, $z_0=z(q_0)$, $a_0=a(q_0)$, etc. If two of the quantities $\alpha_0$, $\tau_0^2$, $z_0$, $a_0$, $b_0$, $c_0$ are algebraic so are all the others. Even more, in that case $\alpha_0$ and $\tau_0^2$ are rational.

\subsection*{Remark}
As one has
\[
\sum_{n=0}^{\infty} \frac{A_{n+x}}{A_x} z^{n+x} = w_0 + w_1 x + w_2 x^2 + w_3 x^3 + w_4 x^4 + O(x^5),
\]
we can write (\ref{expansion}) in the following way:
\begin{multline}\nonumber
\frac{1}{A_x} \sum_{n=0}^{\infty} z^{n+x} A_{n+x}(a+b(n+x)+c(n+x)^2) \\=e^{i \pi r x} \left( \frac{1}{\pi^2}-\alpha x^2 + h \frac{\zeta(3)}{\pi^2}x^3+\frac{\pi^2}{2}(\tau^2-\alpha^2)x^4 \right) + \mathcal{O}(x^5),
\end{multline}
The rational constant $h$ appears (and can be defined) by the coefficient of $x^3$ in the expansion of $A_x$ (analytic continuation of $A_n$) \cite[eq. 4]{AlGu}. In the hypergeometric cases we know how to extend $A_n$ to $A_x$ because the function $\Gamma$ is the analytic continuation of the factorial. To determine $h$ in the non-hypergeometric cases we will not use this definition because it is not clear how to extend $A_n$ to $A_x$ in an analytic way. Instead, we will use the following conjecture

\subsection*{Conjecture}
The radius of convergence $z_c$ of $w_0(z)$ is the smallest root of $P(z)=0$ and
\[ \frac{dz}{dq}(q_c)=0. \]
In addition, $\alpha_c$ is rational. Hence there is a relation with integer coefficients among the numbers $\frac{1}{6} \ln^3|q_c|-T(q_c)$, $\pi^2 \ln|q_c|$ and $\zeta(3)$, which we can discover with the PSLQ algorithm and it determines the invariant $h$. This solution corresponds to the degenerated series $z=z_c$, $c(q_c)=b(q_c)=a(q_c)=0$.

\subsection{New series for $1/\pi^2$}

To discover Ramanujan-like series for $1/\pi^2$ we first obtain the mirror map, the Yukawa coupling and the function $T(q)$. Solving the equation
\[ \frac{dz(q)}{dq}=0 \]
we get the value $q_c$ which corresponds to $z_c$. Let $q=e^t \, e^{i \pi r}$, where $t<0$ is real. If we choose a value of $r$, then we can write (\ref{for-alpha}) in the form
\[ \alpha(t)=\frac{\frac{1}{6}t^3-T(q)-h \zeta (3)}{\pi^2 t}. \]
For $r=0$ we get series of positive terms and for $r=1$ we get alternating series. Solving numerically the equation $\alpha(t)=\alpha_0$, where $\alpha_0$ is rational, we find an approximation of $t_0$ and hence also an approximation of $q_0$. Substituting this $q_0$ in (\ref{for-tau}) we get the value of $\tau_0$. If $\tau_0^2$ is also rational then with the mirror map we get the corresponding approximation of $z_0$. To discover the exact algebraic number $z_0$ we use the Maple function MinimalPolynomial which finds the minimal polynomial of a given degree, then we use the functions $c(q)$, $b(q)$, $a(q)$ to get the numerical values $c_0$, $b_0$ and $a_0$. To recognize the exact algebraic values of these parameters we use MinimalPolynomial again. It is remarkable that in the ``divergent'' cases, we can compute $c_0$, $b_0$, $a_0$ with high precision by using formula (\ref{w0-cua}) for $w_0$.

\subsection*{Big-Table}

In \cite{AlEnStZu} there is a collection of many differential equations of Calabi-Yau type. We select some of those which are pullbacks of differential equations of fifth order. The ones not mentioned below gave no result. The symbol $\#$ stands as a reference of the equation in the Big Table. In \cite{Al1} one can learn the art of finding Calabi-Yau differential equations. In Table \ref{invariants}, we show the invariants corresponding to the cases $\#60$, $\#130$, $\#189$, $\#355$ and, $\#356$.
For all the cases cited above we have found examples of Ramanujan-like series for $1/\pi^2$. In Table \ref{examples} we show those examples, indicating the algebraic values of $\alpha-\alpha_c$, $z_c^{-1} \cdot z$, $a$, $b$ and $c$ for which we have
\[ \sum_{n=0}^{\infty} \widetilde{A}_n \left(z_c^{-1} \cdot z \right)^n (a+bn+cn^2)=\frac{1}{\pi}, \]
where $\widetilde{A_n}=A_n z_c^n$. If $|z_c^{-1} \cdot z | > 1$, then the series diverges but we avoid the divergence considering the analytic continuation given by the parametrization with $q$.

\begin{center}
\begin{table}
\begin{tabular}{|l|}
  \hline \\
  $\quad \textbf{\#60}$ \qquad ${\displaystyle A_n=\sum_{k=0}^{n}\binom{n}{k}^2\binom{2k}{k}\binom{2n-2k}{n-k} \binom{n+k}{n}\binom{2n-k}{n}}$ \\ \\ \hline \\
  $\quad P(z)=(1-16z)^2(1-108z)^2$, \quad
  ${\displaystyle z_c=\frac{1}{2^2\cdot3^3}, \quad \alpha_c=\frac13, \quad \tau^2_c=\frac{2}{23}, \quad h=\frac{50}{23}}$ \\ \\
  \hline \hline \\
  $\quad \textbf{\#130}$ \qquad ${\displaystyle A_n=\sum_{\stackrel{\scriptstyle i+j+k+l}{+m+s=n}} \left( \frac{n!}{i! \, j! \, k! \, l! \, m! \, s!}\right)^2}$ \\ \\ \hline \\
  $\quad P(z)=(1-4z)^2(1-16z)^2(1-36z)^2$, \quad
  ${\displaystyle z_c=\frac{1}{36}, \quad \alpha_c=\frac16, \quad \tau^2_c=\frac{2}{45}, \quad h=\frac{2}{3}}$ \\ \\
  \hline \hline \\
  $\quad \textbf{\#189}$ \qquad ${\displaystyle  A_{n}=\binom{2n}{n}\sum_{j,k}\binom{n}{j}^{2}\binom{n}{k}^{2}\binom{j+k}{n}^2}$ \\ \\ \hline \\
  $\quad P(z)=(1-4z)^2(1-256z)^2$, \quad
  ${\displaystyle z_c=\frac{1}{256}, \quad \alpha_c=\frac12, \quad \tau^2_c=\frac{8}{21}, \quad h=\frac{30}{7}}$ \\ \\
  \hline \hline \\
  $\quad \textbf{\#355}$ \qquad Explicit formula for $A_n$ not known \\ \\ \hline \\
  $\quad P(z)=(1-64z)^2(1-108z)^2$, \quad
  ${\displaystyle z_c=\frac{1}{108}, \quad \alpha_c=\frac13, \quad \tau^2_c=\frac{4}{33}, \quad h=\frac{30}{11}}$ \\ \\
  \hline \hline \\
  $\quad \textbf{\#356}$ \qquad ${\displaystyle A_0=1, \quad A_{n>0}= 2\binom{2n}{n}\sum_{k=0}^{[n/4]}\frac{n-2k}{3n-4k}\binom{n}{k}^2\binom{2k}{k}\binom{2n-2k}{n-k}\binom{3n-4k}{2n}}$ \quad \\ \\ \hline \\
  $\quad P(z)=(1-108z)^2(1-128z)^2$, \quad
  ${\displaystyle z_c=\frac{1}{128}, \quad \alpha_c=\frac13, \quad \tau_c^2=\frac{1}{10}, \quad h=\frac{14}{5}}$ \\ \\
  \hline
\end{tabular}
\vskip 0.5cm
\caption{Table of invariants}\label{invariants}
\end{table}
\end{center}

\begin{center}
\begin{table}
\begin{tabular}{|c|c|c|c|c|c|}
\hline &&&&& \\
$\quad \# \quad$ & $\quad \alpha_0-\alpha_c \quad $ & $\quad z_c^{-1} \cdot z_0 \quad$ & $\quad a_0 \quad$ & $\quad b_0 \quad $ & $\quad c_0 \quad$ \\[1.5ex]
\hline \hline &&&&& \\
$\textbf{60}$ & $\dfrac{4}{23}$ & $\dfrac{1}{2}$ & $\dfrac{3}{3 \cdot 23}$ & $\dfrac{20}{3 \cdot 23}$ & $\dfrac{40}{3 \cdot 23}$ \\[2ex]
$$ & $\dfrac{8}{23}$ & $\dfrac{3^3}{5^3}$ & $\dfrac{40}{5^2 \cdot 23}$ & $\dfrac{282}{5^2 \cdot 23}$ & $\dfrac{616}{5^2 \cdot 23}$ \\[2ex]
$$ & $\dfrac{43}{46}$ & $-\dfrac{1}{48}$ & $\dfrac{706}{2^5 \cdot 3^2 \cdot 23}$ & $\dfrac{5895}{2^5 \cdot 3^2 \cdot 23}$ & $\dfrac{16380}{2^5 \cdot 3^2 \cdot 23}$ \\[2ex]
$$ & $\dfrac{3}{46}$ & $-2$ & $\dfrac{178}{2^5 \cdot 23}$ & $\dfrac{719}{2^5 \cdot 23}$ & $\dfrac{860}{2^5 \cdot 23}$ \\[2ex] \hline &&&&& \\
$\textbf{130}$ & $\dfrac{1}{6}$ & $-\dfrac{3^2}{4^2}$ & $\dfrac{21}{96}$ & $\dfrac{74}{96}$ & $\dfrac{85}{96}$ \\[2ex]
$$ & $0$ & $-4$ & $\dfrac{38}{54}$ & $\dfrac{94}{54}$ & $\dfrac{65}{54}$ \\[2ex] \hline &&&&& \\
$\textbf{189}$ & $\dfrac{2}{21}$ & $\dfrac{8^2}{9^2}$ & $\dfrac{48}{3^5 \cdot 7}$ & $\dfrac{328}{3^5 \cdot 7}$ & $\dfrac{680}{3^5 \cdot 7}$ \\[2ex]
$$ & $\dfrac{4}{7}$ & $\dfrac{1}{3^2}$ & $\dfrac{87}{2^4 \cdot 3^2 \cdot 7}$ & $\dfrac{710}{2^4 \cdot 3^2 \cdot 7}$ & $\dfrac{1840}{2^4 \cdot 3^2 \cdot 7}$ \\[2ex]
$$ & $\dfrac{19}{42}$ & $-\dfrac{2^4}{3^4}$ & $\dfrac{843}{2^2 \cdot 3^5 \cdot 7}$ & $\dfrac{5750}{2^2 \cdot 3^5 \cdot 7}$ & $\dfrac{12610}{2^2 \cdot 3^5 \cdot 7}$ \\[2ex]
$$ & $\dfrac{139}{42}$ & $-\dfrac{2^4}{21^4}$ & $\quad \dfrac{1655799}{2^2 \cdot 3^5 \cdot 7^5} \quad$ & $\quad \dfrac{24749870}{2^2 \cdot 3^5 \cdot 7^5} \quad$ & $\quad \dfrac{122761930}{2^2 \cdot 3^5 \cdot 7^5} \quad$ \\[2ex]
$$ & $\dfrac{1}{14}$ & $-\dfrac{4^2}{3^2}$ & $\dfrac{51}{252}$ & $\dfrac{254}{252}$ & $\dfrac{370}{252}$ \\[2ex] \hline &&&&& \\
$\textbf{355}$ & $\dfrac{1}{11}$ & $\dfrac{3}{4}$ & $\dfrac{1}{2^2 \cdot 3 \cdot 11}$ & $\dfrac{12}{2^2 \cdot 3 \cdot 11}$ & $\dfrac{30}{2^2 \cdot 3 \cdot 11}$ \\[2ex]
$$ & $\dfrac{5}{22}$ & $-\dfrac{3^2}{4^2}$ & $\dfrac{9}{2^2 \cdot 11}$ & $\dfrac{42}{2^2 \cdot 11}$ & $\dfrac{60}{2^2 \cdot 11}$ \\[2ex]
$$ & $\dfrac{5}{11}$ & $\dfrac{27}{196}$ & $\dfrac{21}{2^2 \cdot 7 \cdot 11}$ & $\dfrac{164}{2^2 \cdot 7 \cdot 11}$ & $\dfrac{390}{2^2 \cdot 7 \cdot 11}$ \\[2ex]
$$ & $\dfrac{13}{11}$ & $\dfrac{1}{108}$ & $\dfrac{3119}{2^2 \cdot 3^6 \cdot 11}$ & $\dfrac{29860}{2^2 \cdot 3^6 \cdot 11}$ & $\dfrac{93090}{2^2 \cdot 3^6 \cdot 11}$ \\[2ex]
$$ & $\dfrac{1}{22}$ & $-3$ & $\dfrac{16}{3 \cdot 11}$ & $\dfrac{60}{3 \cdot 11}$ & $\dfrac{60}{3 \cdot 11}$ \\[2ex] \hline &&&&& \\
$\textbf{356}$ & $\dfrac{1}{5}$ & $\dfrac{1}{2}$ & $\dfrac{2}{160}$ & $\dfrac{27}{160}$ & $\dfrac{74}{160}$ \\[2ex]
$$ & $1$ & $\dfrac{1}{50}$ & $\dfrac{74}{800}$ & $\dfrac{679}{800}$ & $\dfrac{2002}{800}$ \\[2ex]
$$ & $\dfrac{7}{10}$ & $-\dfrac{1}{2^4}$ & $\dfrac{158}{1280}$ & $\dfrac{1113}{1280}$ & $\dfrac{2618}{1280}$ \\[2ex] \hline
\end{tabular}
\vskip 0.5cm
\caption{Table of examples}\label{examples}
\end{table}
\end{center}
For $\#355$ an explicit formula for $A_n$ is not known but we can easily compute these numbers from the fifth order differential equation $Dw=0$, where $D$ is the following operator:
\begin{multline}\nonumber
\theta^5-2z(2\theta+1)(43\theta^4+86\theta^3+77\theta^2+34\theta+6) \\ + 48z^2(\theta+1)(2\theta+1)(2\theta+3)(6\theta+5)(6\theta +7).
\end{multline}

\subsection*{Complex series for $1/\pi^2$}

Another method to obtain series for $1/\pi^2$ is by applying suitable transformations to the already known series for $1/\pi^2$; see \cite{Zu1}, \cite{Al5} and \cite{AlGu}. Although we can use this technique to obtain other real Ramanujan-like series for $1/\pi^2$ our interest here is to find examples of Ramanujan-like complex series for $1/\pi^2$. For that purpose we will use the following very general transformation:
\[
\sum_{n=0}^{\infty} A_n z^n = \frac{1}{1-z} \sum_{n=0}^{\infty} a_n \left[ u \left( \frac{z}{1-z} \right)^m  \right]^n, \quad
A_n=\sum_{k=0}^{n} u^k \binom{n}{mk} a_k,
\]
where $u=1$ or $u=-1$ and $m$ is a positive integer (check that both sides satisfy the same Calabi-Yau differential equation). For example, translating the hypergeometric series
\begin{equation}\label{series-252-pi2}
\sum_{n=0}^{\infty} \frac{(3n)!(4n)!}{n!^7} (252n^2+63n+5) (-1)^n \left( \frac{1}{24} \right)^{4n}=\frac{48}{\pi^2},
\end{equation}
taking $u=-1$ and $m=4$, we find four series, one of them is the complex series
\begin{equation}\label{ex-complex-pi2}
\sum_{n=0}^{\infty} A_n \left( \! \frac{}{} 9072n^2 +(9072-756i)n + (2875-516i) \right) \left( \frac{1}{1-24i} \right)^n=\frac{27504+3454i}{\pi^2},
\end{equation}
where
\[ A_n=\sum_{k=0}^n (-1)^k \binom{n}{4k} \frac{(3k)! (4k)!}{k!^7}. \]
Transformations preserve the value of the invariants $h$, $\alpha_c$ and $\tau_c$ and the series (\ref{ex-complex-pi2}) has $2(\alpha-\alpha_c)=3$ and $\tau=3 \sqrt{3}$ because it is a transformation of (\ref{series-252-pi2}); see \cite{AlGu}.
Looking at the transformation with $u=-1$ and $m=4$, we see that the mirror maps $z$ and $z'$ corresponding to $A_n$ and $a_n$, are related in the following way:
\[ z=\frac{\sqrt[4]{z'}}{1+\sqrt[4]{z'}}. \]
Write $z=z(q)$ and $z'=z'(q')$. Then, the first terms of $J(q)$ are
\[ J(q) = \frac{1}{q}+1+582q^3+277263q^7+167004122q^{11}+\cdots, \]
with $q=\sqrt[4]{q'}$. Writing, as usual, $q=e^{i \pi r} |q|$, we deduce that as the series (\ref{series-252-pi2}) has $r=1$ then the series (\ref{ex-complex-pi2}) has $r=1/4$.

\section{ADDENDUM}

The method used in this paper, to find $h$, $\alpha_c=\alpha(q_c)$ and $\tau_c=\tau(q_c)$, is valid for those Calabi-Yau differential equations such that $K(q_c)=0$, where $q_c$ is a solution of $dz/dq=0$. In these cases, we conjecture that $z(q_c)$ is the smallest root of $P(z)$ and that $a(q_c)=b(q_c)=c(q_c)=0$. But from Th. \ref{teor-abc}, we see that $b(q_c)=0$ implies that $\tau_c=f(q_c)$, where
\[ f(q)=\frac{1}{\pi^2} \left( \frac{}{} \! \! \theta_q^2 T(q)-\ln|q| \right) \frac{L(q)}{\theta_q L(q)}, \]
which allows us to obtain the critical value of $\tau$. Then, replacing $q=q_c$
in (\ref{for-tau}), we can obtain $\alpha_c$. Finally, replacing $q=q_c$ in (\ref{for-alpha}), we obtain the value of $h$. As $q_c>0$, the formula for $h$ can be written in the form $h=h(q_c)$, where
\[ h(q)=\frac{1}{\zeta(3)} \left( \Phi(q)-\ln (q) \, \theta_q \Phi(q) - \ln (q) \, \frac{L(q)}{\theta_q L(q)} \theta^2_q \Phi(q) \right), \]
and $\Phi(q)$ is the Gromov-Witten potential (\ref{GW}). The advantage of this way of getting the invariants $\tau_c$, $\alpha_c$ and $h$ is that we use explicit formulas instead of the PSLQ algorithm.


\begin{thebibliography}{99}

\bibitem{Al2}
G.Almkvist,
\textit{Calabi-Yau differential equations of degree $2$ and $3$ and Yifan Yang's pullback}, (2006);
(arXiv: math/0612215).

\bibitem{Al1}
G. Almkvist,
\textit{The art of finding Calabi-Yau differential equations},
in Gems in Experimental Mathematics T. Amdeberhan, L.A. Medina, and V.H. Moll (eds.),
Contemp. Math. \textbf{517} (2010), Amer. Math. Soc., 1--18;
(arXiv:0902.4786).

\bibitem{Al5}
G. Almkvist,
\textit{Transformations of Jesús Guillera's formulas for $1/\pi^2$}, (2009); \\
(arXiv:0911.4849).

\bibitem{Al6}
G. Almkvist,
\textit{Some conjectured formulas for $1/\pi$ coming from polytopes, K3-surfaces and Moonshine, manuscript}.

\bibitem{AlEnStZu}
G.Almkvist, C.van Enckevort, D.van Straten, W.Zudilin,
\textit{Tables of Calabi-Yau equations}, (2005).
(arXiv: math/0507430).

\bibitem{AlGu}
G. Almkvist and J. Guillera,
\textit{Ramanujan-like series and String theory},
Exp. Math. \textbf{21}, (2012), 223-234.
(eprint arXiv:1009.5202)

\bibitem{AlZu}
G. Almkvist and W. Zudilin,
\textit{Differential equations, mirror maps and zeta values},
in Mirror Symmetry V, N. Yui, S.-T. Yau, and J.D. Lewis (eds.), AMS/IP Studies in Advanced Mathematics 38 (2007), International Press \& Amer. Math. Soc., 481--515;
(e-print math.NT/0402386)

\bibitem{BaBeCh}
N.D. Baruah, B.C. Berndt, H.H. Chan,
\textit{Ramanujan's series for $1/\pi$: A survey},
The Amer. Math. Monthly \textbf{116} (2009) 567-587.; available at Bruce Berndt's web-site.

\bibitem{Bo}
J.M. Borwein, P.B. Borwein,
\textit{Pi and the AGM: A Study in Analytic Number Theory and Computational Complexity},
(Canadian Mathematical Society Series of Monographs and Advanced Texts), Jonh Wiley, New York, (1987).

\bibitem{chan2}
H.H. Chan, S.H. Chan and Z. Liu,
\textit{Domb's numbers and Ramanujan-Sato type series for $1/\pi$},
Adv. Math. \textbf{186} (2004) 396-410.

\bibitem{ChWaZu}
H.H. Chan, J. Wan, W. Zudilin,
\textit{Complex series for $1/\pi$},
Ramanujan J. (to appear).

\bibitem{ChCh}
D. Chudnovsky and G. Chudnovsky
\textit{Approximations and complex multiplication according to Ramanujan},
In {\it Ramanujan Revisited: Proceedings of the Centenary Conference, University of Illinois at Urbana-Champaign}, G. Andrews, R. Askey, B. Berndt, K. Ramanathan, and R. Rankin, eds., Academic Press, Inc., Boston, MA, 1987, 375-472.

\bibitem{CoKa}
D.A.Cox and S.Katz,
\textit{Mirror symmetry and algebraic Geometry},
AMS, Providence, (1999).

\bibitem{Gu5}
J. Guillera,
\textit{A matrix form of Ramanujan-type series for $1/\pi$};
in Gems in Experimental Mathematics T. Amdeberhan, L.A. Medina, and V.H. Moll (eds.),
Contemp. Math. \textbf{517} (2010), Amer. Math. Soc., 189--206;
(arXiv:0907.1547).

\bibitem{Gu7}
J. Guillera,
\textit{Collection of Ramanujan-like series for $1/\pi^2$}. Unpublished manuscript available at J. Guillera's web site.

\bibitem{Gu6}
J. Guillera,
\textit{Mosaic supercongruences of Ramanujan-type}.
Exp. Math. \textbf{21}, (2012), 65-68.
(e-print arXiv:1007.2290).

\bibitem{GuiZu}
J. Guillera and W. Zudilin,
\textit{``Divergent'' Ramanujan-type supercongruences}.
Proc. Amer. Math. Soc. 140:3 (2012), 765--777.
(e-print arXiv:1004.4337)

\bibitem{tito}
T. Piezas,
Personal communication.
\textit{Ramanujan-type complex series available at Tito Piezas's web-site}.

\bibitem{Ra}
S. Ramanujan,
\textit{Modular equations and approximations to $\pi$},
Q. J. Math. \textbf{45} (1914), 350--372.

\bibitem{yangzudilin}
Y. Yang and W. Zudilin,
\textit{On $\operatorname{Sp}_4$ modularity of Picard--Fuchs differential equations for Calabi--Yau threefolds}, (with an appendix by V.~Pasol); in Gems in Experimental Mathematics T. Amdeberhan, L.A. Medina, and V.H. Moll (eds.),
\emph{Contemp. Math.} \textbf{517} (2010), Amer. Math. Soc., 381--413;
(arXiv:0803.3322).

\bibitem{Zu1}
W. Zudilin,
\textit{Quadratic transformations and Guillera's formulae for $1/\pi^2$},
Mat. Zametki \textbf{81}:3 (2007), 335--340;
English transl.,
Math. Notes \textbf{81}:3 (2007), 297--301. \\
(arXiv:math/0509465v2).

\bibitem{Zu3}
W. Zudilin,
\textit{Ramanujan-type supercongruences},
J. Number Theory \textbf{129}:8 (2009), 1848--1857.
(arXiv:0805.2788).

\bibitem{Zu4}
W. Zudilin,
\textit{Arithmetic hypergeometric series},
Russian Math. Surveys 66:2 (2011), 369--420.
Russian version in Uspekhi Mat. Nauk 66:2 (2011), 163--216.

\end{thebibliography}
\end{document}